\documentclass[12pt, reqno]{amsart}
\allowdisplaybreaks[1]
\usepackage{amsmath}
\usepackage{amssymb}
\usepackage{amsfonts}
\usepackage[usenames]{color}
\makeindex
\usepackage{hyperref}
\usepackage{bookmark}

 \newtheorem{theorem}{Theorem}[section]
 \newtheorem{corollary}[theorem]{Corollary}
 \newtheorem{lemma}[theorem]{Lemma}

\theoremstyle{definition}

\theoremstyle{remark}
\newtheorem{rem}[theorem]{Remark}

\newcommand\dd{\mathrm d}

\newcommand{\C}{\mathbb{C}}
\newcommand{\U}{\mathcal{U}}

\newcommand{\R}{\mathbb{R}}

\newcommand{\cc}[1]{\overline{#1}}
\newcommand{\abs}[1]{\left\vert#1\right\vert}

\newcommand{\norm}[1]{\left\Vert#1\right\Vert}

\newcommand{\ad}{^\ast}

\newcommand{\inv}{^{-1}}

\newcommand{\til}{\raise.17ex\hbox{$\scriptstyle\mathtt{\sim}$}}

\newcommand{\ph}{\varphi}

\newcommand\la{\lambda}

\newcommand\beq{\begin{equation}}

\newcommand\eeq{\end{equation}}

\newcommand\black{\color{black}}

\newcommand\bbm{\begin{bmatrix}}
\newcommand\ebm{\end{bmatrix}}
\newcommand\bpm{\begin{pmatrix}}
\newcommand\epm{\end{pmatrix}}
\numberwithin{equation}{section}

\newlength{\Mheight}
\newlength{\cwidth}

\newcommand{\SC}{\mathbb{H}(\C)^g}
\newcommand{\SCk}{\mathbb{H}\left(\mathbb{C}^{\kappa \times \kappa}\right)}

\newcommand{\dfn}[1]{{\bf #1}\index{#1}}

\newcommand{\phax}{\ph^{A, X}_v}

\title[Convex entire noncommutative functions]{Convex entire noncommutative functions are polynomials of degree two or less}

\author[J.W. Helton]{J. William Helton$^1$}

\address{ J. William Helton, Mathematics
Department \\ University of California at San Diego 92093}
\email{helton@math.ucsd.edu}
\thanks{$^1$ Partly supported by NSF grant  DMS-1201498  and the Ford Motor Co.}

\author[J. E. Pascoe]{J. E. Pascoe$^2$}
\address{ J. E. Pascoe, Mathematics
Department \\ University of California at San Diego 92093}
\email{jpascoe@math.ucsd.edu}
\thanks{$^2$ Partially supported by NSF grant DMS-1361720}

\author[R. Tully-Doyle]{Ryan Tully-Doyle}
\address{ R. Tully-Doyle, Mathematics
Department \\ University of California at San Diego 92093}
\email{rtullydo@gmail.com}

\author[V. Vinnikov]{Victor  Vinnikov$^3$ }
\address{V. Vinnikov, Department of
Mathematics\\ Ben Gurion Univ. of the Negev, Beer Sheva\\Israel}
\email{vinnikov@math.bgu.ac.il}
\thanks{$^3$ Partially supported by the Israel Science Foundation (Grant No.~322/00)}

\date{\today}

\begin{document}

\begin{abstract}

This paper concerns matrix ``convex'' functions of (free) noncommuting variables, $x = (x_1, \ldots, x_g)$.  It was shown in \cite{helmconvex04} that a polynomial in $x$ which is matrix convex is of degree two or less. We prove a more general result: that a function of $x$ that is matrix convex near $0$ and also that is ``analytic'' in some neighborhood of the set of all self-adjoint matrix tuples is in fact a polynomial of degree two or less.

More generally, we prove that a function $F$ in two classes of noncommuting variables, $a = (a_1, \ldots, a_{\tilde{g}})$ and $x = (x_1, \ldots, x_g)$ that is ``analytic'' and matrix convex in $x$ on a ``noncommutative open set'' in $a$ is a polynomial of degree two or less.

\black
\end{abstract}

\maketitle

\section{Introduction}

Helton and McCullough in \cite{helmconvex04} showed that if a noncommuting polynomial $p(x)$ in noncommuting 
(free) variables is matrix convex, then $p$ has degree two or less. In \cite{hhlm08}, Hay, Helton, Lim, and McCullough showed that a noncommuting polynomial $p(a,x)$ in two classes of noncommuting variables that is matrix convex (in $x$) has degree two or less. 
 The main result of this paper, Theorem \ref{main}, implies this conclusion for classes of convex noncommuting functions far more general than polynomials. In particular, we show this phenomenon holds for noncommuting  (free) ``entire'' functions.

Unlike other proofs of similar results which use either the positivity of the Hessian (combined with an appropriate sums of squares theorem), or a direct analysis of the Gram matrix, or a realization as some kind of transfer function, the approach here is to reduce the problem to the one variable situation by working on a slice and to exploit the classical integral representations for matrix convex and matrix monotone functions of one variable.

\subsection{Classical one variable functions}

In one variable, a self-adjoint matrix $B$ is diagonalizable and so can be written in the form
	$$B = U\ad \bpm \la_1 &  & &  \\  & \la_2 & & \\ & & \ddots &  \\  &  &  & \la_n \epm U.$$

Then we define evaluation of a function $f:\C \to \C$ at the matrix $B$ by
	$$f(B) = U\ad \bpm f(\la_1) &  &  &  \\  & f(\la_2) &  &  \\ & & \ddots &  \\  &  &  & f(\la_n) \epm U.$$

A function $f: (a,b) \rightarrow \mathbb{R}$ is called \dfn{matrix convex} if for any pair of Hermitian matrices of the same size $A$ and $B,$
	$$f(tA +(1-t)B)\preceq tf(A) + (1-t)f(B)$$
for all $t \in [0,1].$

The work of F. Kraus  \cite{kraus36} on functions of one variable led to the observation that classical matrix-convex functions $f: [-1,1] \rightarrow \mathbb{R}$ have the following representation for some probability measure $\mu$ on $[-1,1]:$
	\beq \label{krauseq} f(t) = f(0) + f'(0)t + \frac{1}{2}f''(0)\int^1_{-1} \frac{t^2}{1-\lambda t} d\mu(\lambda). \eeq
Conversely, any $f$ of the form \eqref{krauseq} must be matrix-convex.
Notably, the notion of matrix convexity is much stronger that the notions
of convexity and concavity in freshman calculus' for example, the function $f(x) = x^4$
is not matrix convex, as we will see later.

However, whenever a convex function is taken to be real entire, i.e. the function has an analytic continuation off the interval $[1,1]$ to an open set containing the real axis in $\C$, this puts stringent conditions on the support of the measure $\mu$ in representation \eqref{krauseq}. In fact, {\it a matrix convex function that is real entire is a polynomial of degree $2$ or less} (Lemma \ref{onedimension} below). In this article, we shall extend this to ``entire functions'' in several noncommuting variables.

\subsection{Non-quadratic matrix convex functions}

We note that there are functions more general than polynomials which are matrix convex in some region. For example, as a consequence of \eqref{krauseq} with the measure $\mu$ given by a point mass at $\frac{1}{2}$, the representation reduces to
	$$f(t) = \frac{t^2}{1- \frac{1}{2}t}.$$
In this case, $f$ is a function that is matrix-convex on matrices $-I\preceq X \preceq I$. There are many other examples in \cite{bha97} of functions which are matrix convex on intervals. Also, in several noncommuting variables, there are many nc  rational functions which are convex in a region. These are completely classified in \cite{helmcvin06}.  In another direction, variants of logarithms related to entropy and relative entropy are (even in several variables) matrix convex on positive matrices \cite{eff09, neg11}. Applications and new proofs can be found in \cite{auj11, neg13}. However, in keeping with our results here, none of these examples can extend to be real entire and thus each exhibits some kind of singular behavior.

\subsection{Noncommuting Variables, Sets and Polynomials}

This article is concerned with matrix convex functions in several noncommuting  variables, to be defined later. To do this, we will need to define and describe functions and sets in noncommutative variables. Often we will abbreviate the term noncommutative by \dfn{nc}.

\subsubsection{NC variables and NC polynomials}

Let $z = (z_1 \, \ldots, z_g)$ denote a tuple of noncommuting variables. 
Let $\C\langle z \rangle$ denote the nc polynomials 
in the variables $z_1 \, \ldots, z_g.$
For example, $z_1z_2 + 7z_2z_1 + z_1^2$ and $z_2z_1 + 7z_2z_1+ z_1^2$ 
are nc polynomials
in $z_1$ and $z_2$ that are not equal since the variables do not
commute. 
A polynomial $p \in \C\langle z \rangle$ will be denoted by
	$$p(z) = \sum_w^{\text{finite}} p_w z^w,$$
over words $z^w$ in the letters $z$ with coefficients $p_w \in \C$. That is,
words are the multi-indicies in the nc setting.

There is a natural involution $\ad$ on the ring of nc polynomials 
$\C\langle z \rangle$ that reverses the order of a word. 
The involution is defined as follows: 
\begin{enumerate}
	\item 
	\label{it:hermvar}
		$z_i^* =z_i  $ for all $i = 1, \ldots, \tilde{g}$,
	\item $(pq)\ad = q\ad p\ad$ for $p, q \in \C\langle z \rangle$,  
	\item and  $(\lambda p)^* = \cc{\lambda}p^*$ for any
	  $\lambda\in\mathbb{C}$.
\end{enumerate}

The variables $z_i$ are called \dfn{Hermitian variables} because of 
condition \eqref{it:hermvar}
and they are called \dfn{free}
  because there
is no polynomial identity which they \emph{a priori} satisfy.

For example,
	$$(iz_1z_2 + 7z_2z_1 + z_1^2)^* = -z_2z_1 + 7 z_1z_2 + z_1^2.$$

A polynomial $p \in \C\langle z\rangle$ is a \dfn{Hermitian noncommuting polynomial} if
	$p\ad = p.$
For example, the polynomial
	$$8z_1z_2 + 8z_2z_1+z_1^2+z_2^{81} \text{ is Hermitian,}$$ but
$$8z_1z_2 + 6z_2z_1+z_1^2+z_2^{81} \text{ is not Hermitian.}$$

A polynomial $p$ in $g$ noncommuting variables $z = (z_1, \ldots, z_g)$ can be evaluated at tuples of hermitian matrices. Let $\mathbb H(\C^{n\times n})^{g}$ denote the set of Hermitian $n \times n$ matrices. For a word $z^w = z_{j_1}z_{j_2}\ldots z_{j_m}$ and a $g$-tuple of Hermitian matrices $Z = (Z_1, \ldots, Z_g) \in \mathbb H(\C^{n\times n})^{g}$, we define
	$$Z^w = Z_{j_1}Z_{j_2}\ldots Z_{j_m}$$
and for a polynomial $p \in \C\langle z \rangle$,
	$$p(Z) = \sum_{w}^{\text{finite}} p_w Z^w.$$
Note that for a Hermitian polynomial $p$, 
	$$p(Z) = p\ad(Z) = p(Z)\ad,$$
that is, we say that $p$ is {\bf evaluation Hermitian}.

A matrix-valued polynomial $p$ is a function
	$$p(z) = \bpm p_{11}(z) & \cdots & p_{1m}(z) \\ \vdots & \ddots & \vdots \\ p_{m1}(z) & \cdots & p_{mm}(z)\epm $$
where each entry $p_{ij} \in \C\langle z \rangle$. A matrix-valued polynomial $p$ is Hermitian if $p\ad = p$.
\black

\subsection{ NC Sets and NC Functions}\label{sec:ncset}

	\def\bbH{{\mathbb H}}
	Let $\mathbb H (\C^{n\times n})^g$ denote $g$-tuples of 
	self-adjoint 
	(also called Hermitian)
	$n \times n$ matrices over $\C$, that is
	$$\mathbb H (\C^{n\times n})^g = \{ X = (X_1, \ldots, X_g) \in (\C^{n\times n})^g : X_i = X_i\ad, i = 1, \ldots g\},$$
and let $\mathbb H (\C)^g$ be the (self-adjoint) matrix universe
	$$\mathbb H (\C)^g = \bigcup_{n = 1}^\infty \mathbb H (\C^{n\times n})^g.$$
A set $D \subset \SC$ is called a
 \dfn{unitarily invariant nc set} if it satisfies the following conditions.
 We shall only use unitarily invariant sets in this paper, so
 henceforth we refer to them as \dfn{nc sets}.

	\begin{description}
	\item[$D$ is closed with respect to direct sums] 
	\hfill 	
	
	if $Z = (Z_1, \ldots, Z_g) \in D$ and $\tilde{Z} = (\tilde{Z}_1, \ldots, \tilde{Z}_g) \in D$ then
	$$
		\bpm Z &  \\  & \tilde{Z} \epm = \left(\bpm Z_1 &  \\  & \tilde{Z}_1 \epm, \ldots, \bpm Z_g &  \\ & \tilde{Z}_g \epm\right) \in D.
	$$
	\item[$D$ is closed with respect to unitary equivalence]
	\hfill

	If $Z \in D\bigcap \mathbb H(\C^{n\times n})^g, U \in U_n$, then
	\[
	U\ad Z U = (U\ad Z_1 U, \ldots, U\ad Z_g U) \in D.
	\]
	\end{description}
Here, $U_n$ denotes the unitary matrices of size $n$. 

We define a norm $\norm{\cdot}$ on each set $\mathbb H(\C^{n\times n})^{g}$ by
	$$\norm{(X_1, \ldots, X_g)} = \text{max eigenvalue }(\sum_{i=1}^g X_i X_i\ad)^{1/2}.$$

Now we discuss nc functions.
Let $D$ be an nc set. Let $f: D \to \mathbb M(\C^{n\times n})$ be a function. 
We say that $f$ is an \dfn{nc function} if it satisfies the following conditions.
	\begin{description}
	\item[$f$ is graded] If $Z \in D \bigcap \mathbb H(\C^{n\times n})^g$, then $f(Z) \in \mathbb M(\C^{n\times n})$.
	\item[$f$ respects direct sums] If $Z, \tilde{Z} \in D$ then
	\[ 
	f \bpm Z &  \\  & \tilde{Z} \epm = \bpm f(Z) &  \\  & f(\tilde{Z}) \epm.
	\]
	\item[$f$ respects unitary equivalence] 
	\hfill

	If $U \in U_n,$ and $Z \in D\bigcap\mathbb H(\C^{n\times n})^g$, 
	\[
	f(U\ad Z U) = U\ad f(Z) U.
	\]
	\end{description}

A function $F$ on an nc set $D$ is a \dfn{matrix-valued nc function} if
	$$F(z) = \bbm F_{11}(z) & \cdots & F_{1N}(z) \\ 
	\vdots & \ddots & \vdots \\ F_{M1}(z) & \cdots & F_{MN}(z) \ebm,$$
where each $F_{ij}(z)$ is an nc function. $F$ is called \dfn{evaluation Hermitian} if for all matrix tuples $Z \in D$, we have $F(Z)\ad = F(Z)$. Note that Hermitian nc polynomials are examples of nc evaluation Hermitian functions.

\begin{rem} 
In \cite{vvw12}, nc functions were defined differently in that their domain was an nc set in the non-self-adjoint matrix universe, and they were required to respect any similarity that leaves the point inside the domain; see also the fully matricial functions of \cite{voi04, voi10}. However, if $Z \in \mathbb H(\C^{n\times n})^g$ and $T \in GL_n$ so that $TZT\inv \in \mathbb H(\C^{n\times n})^g$, then one can expect $T$ to be a unitary (up to a scalar). Hence for functions defined on nc sets consisting of $g$-tuples of self-adjoint matrices, as we consider here, there is no great difference between respecting similarity and respecting unitary equivalence. 
\end{rem}

\subsection{Functions in $a$ and $x$}

We now introduce an extra level of generality, 
in light of the engineering motivation for this subject (c.f. \cite{hmpv09}). 
Let $z = (a,x)$ be a $(\tilde{g} + g)$-tuple of Hermitian variables,
 so that $a = (a_1, \ldots, a_{\tilde{g}})$ and $x = (x_1 \ldots, x_g)$. Note that nc sets 
$D \subset \mathbb H(\C)^{\tilde{g}} \times \mathbb H(\C)^g$ 
and nc functions $F(a,x)$ have the same definitions as in Section \ref{sec:ncset} by simply thinking of $(a,x)$ as $z = (a, x)$. 

We shall define nc functions $F$ of the tuples of variables $a$ and $x$, where the function is convex in $x$ on a small set and real entire in the variable $x$ but not necessarily in $a$. To do so, we require appropriate notions of real entirety and real matrix convexity.

For a matrix tuple $A \in \SCk^{\tilde{g}}$, 
the \dfn{nc smallest A-set}
is the set $C_A \subset \mathbb H(\C^{\tilde{g}})$ with elements of the form
		$$\alpha = U\ad \bbm A & 0 & 0 \\ 0 & A & 0 \\ 0 & 0 & \ddots \ebm U,$$
	that is 
		\beq\label{zclosure}C_A = \bigcup_{m=1}^{\infty} \{U\ad(I_m \otimes A)U: U \in  U_{\kappa m}\},\eeq
where $U_{\kappa m}$ is the unitary matrices of dimension $\kappa m$.
The {nc smallest $A$-set} is the smallest nc set containing $A$
\footnote{A related notion might be thought of as the nc Zariski closure of a point $A$,
but this set can be bigger than $C_A$ and is not used in the paper.}
Alternatively, one can think of it as the nc set generated by $A$.

Note that $I_m \otimes A$ and $A \otimes I_m$ are permutation similar, 
and so for all $m$, we have $A \otimes I_m \in C_A$.
We denote by $(C_A, 0)$ the set 
	\beq (C_A,0) = \{(\alpha, 0): \alpha \in C_A\}.\eeq

\subsubsection{Analyticity}\label{analytic}
The general theory of nc analytic functions was initiated by J.L. Taylor in \cite{tay73}.
Subsequently, several notions of analyticity for nc functions have been studied, notably \cite{voi04, voi10, vvw12}.
Here we define nc analytic using nc power series near a point and ordinary commutative analytic 
continuation away from the point, as in \cite{helkm11}.

A formal \dfn{$x$-power series} is a sum
		\beq \sum_i F_i(a,x), \eeq
	where each $F_i$ is a matrix-valued nc polynomial homogenous in $x$ of degree $i$. 

For a matrix-valued nc function $F$ with a power series
	$$F(a, x) = \sum_i F_i(a, x)$$
that converges absolutely and uniformly on some nc set $D$, say that $F$ is \dfn{Hermitian} if each $F_i$ is a Hermitian matrix-valued polynomial. Note that $F$ is Hermitian implies that $F$ is evaluation Hermitian on $D$, that is
	$$F(A, X)\ad = F(A,X).$$

Let $A$ be a tuple $A \in \SCk^{\tilde{g}}$, and let $C_A$ be the nc smallest $A$-set. An nc function $F$ is called \dfn{$x$-real entire at $(C_A,0)$} if
	\begin{enumerate}
		\item \label{Xanal} there exists an $\varepsilon > 0$ and an $x$-power series $$F(a,x) = \sum_i F_i(a,x)$$ such that for each $\alpha = U\ad(I_m\otimes A)U \in C_A$,  
			\beq\label{maybe} F(\alpha, X) = \sum_i F_i(\alpha, X),\eeq
 converges on the open $x$-ball $$\{(\alpha, X): X\in \mathbb H(\C^{\kappa m \times \kappa m})^g, \norm{X} < \varepsilon\},$$
		\item\label{Xent}  and the matrix-valued analytic function $F$ in the $x$-ball near $(\alpha, 0)$ analytically continues as a function of $g(\kappa m)^2$ complex variables to an open set in $(\C^{\kappa m\times \kappa m})^g$ containing $\mathbb H(\C^{\kappa m \times \kappa m})^g$. 
	\end{enumerate}

\vspace{15pt}

An $M\times N$ matrix-valued nc function $F$ is $x$-real entire if each entry $F_{ij}$ is $x$-real entire. If each $F_{ij}$ has an $x$-power series in an $x$-ball of radius $\varepsilon_{ij}$ near $(C_A, 0)$, then for any $\alpha \in C_A$, $F$ has an $x$-power series
	$$F(\alpha, X) = \sum_i F_i(\alpha, X)$$
on an $x$-ball of radius $\varepsilon = \min_{i,j} \varepsilon_{ij}$, where each $F_i$ is a matrix-valued nc polynomial whose entries are homogeneous of degree $i$ in $x$. 

\begin{rem} 
	The notion of $x$-real entire requires two rather different notions of analyticity, \eqref{Xanal} and \eqref{Xent}. In our forthcoming proofs, we shall indicate when each notion is used. 
\end{rem}

\begin{rem}
In the case that there are only $x$ variables, our notion of $x$-real entire translates into the language of \cite[Section 7.1]{vvw12} as follows. Consider a matrix-valued nc function $F$ on the $x$-ball $\{X\in \SC: \norm{X} < \varepsilon\}$ in $\SC$ which extends to a locally bounded nc function (in the sense of \cite{vvw12}, i.e. respecting similarity that leaves the point in the domain) on an open (i.e. open in every matrix size) nc set $\Omega$ in the non-self-adjoint matrix universe. If $\Omega$ contains $\SC$, then $F$ is real entire at 0 in the sense of our definition.
\end{rem}

\subsubsection{Convexity}

 Now we introduce the notion of matrix convexity on a ball. Denote by $B_\varepsilon$ the $x$-ball of radius $\varepsilon$ 
	$$\mathcal B_\varepsilon = \{X \in \SC |\norm{X} < \varepsilon\},$$
where $\norm{X}$ is the operator norm on matrix tuples.
For each $n$, denote the $x$-ball on level $n$ by 
$$\mathcal B_{n,\varepsilon} = \mathcal B_\epsilon \cap \SCk^{\tilde{g}}.$$

 Following \cite{hhlm08}, we define a notion of partial convexity in $x$ for a function $f(a,x)$. Denote by $\U_\kappa$ a set in $\SCk^g$ containing $0$.
 Given a matrix $A \in \SCk^{\tilde{g}}$ and a set $\U_{\kappa}$, the set $\{A\} \times \U_{\kappa}$ is \dfn{$x$-open at $(A, 0)$} if $\U_{\kappa}$ contains a ball $\mathcal B_{\kappa, \varepsilon}$.

 For a matrix $A \in \SCk^{\tilde{g}}$, suppose that a set $\{A\}\times \U_\kappa$ is $x$-open near $(A,0)$.  A function $F(a,x)$ is \dfn{matrix convex in $x$} on $\{A\}\times \U_\kappa$ if $F$ is Hermitian and for all $X, Y \in \U_\kappa$ for which
	$$(A,tX) + (A,(1-t)Y) \in \U_\kappa \text{ for } 0\leq t\leq 1,$$
it follows that 
	$$F(A, tX + (1-t)Y)  \preceq tF(A,X) + (1-t)F(A,Y).$$

A function $F$ is \dfn{matrix convex in $x$ at $(A, 0)$} if there exists $\varepsilon > 0$ such that $F$ is matrix convex in $x$ on the set $\{A\}\times \mathcal B_{\kappa, \varepsilon}$. 
A function $F$ is \dfn{matrix convex in $x$ at $(C_A,0)$} if there exists a uniform $\varepsilon > 0$ such that for each $\alpha \in C_A$, $F$ is matrix convex in $x$ at $(\alpha, 0)$. That is, for each $\alpha = U\ad (I_m \otimes A) U$, $F$ is matrix convex on the set
	\beq\label{xball}\{(\alpha, X): X \in \mathbb H(\C^{\kappa m \times \kappa m})^g, \norm{X} < \varepsilon\}.\eeq

We emphasize that when we fix $A$, we also fix the dimension $\kappa$, which consequently determines the size of $X$. As a tuple $\alpha = U\ad(I_m \otimes A)U$ has dimension $\kappa m \times \kappa m$, consequently the tuples $X$ in \eqref{xball} have dimension $\kappa m \times \kappa m$.

Balasubramanian and McCullough proved that ``quasi-convex" noncommuting polynomials have degree two or less in \cite{balmc12}. Our results do not imply this. While this article discusses matrix convex functions, 
there is also work on nc convex sets. When an nc  convex set has an 
nc  polynomial defining function, it has a very rigid form (see \cite{heltmc12}).

\subsection{Main Result}

The main result of this paper is: 
\begin{theorem}\label{main}
Let $A$ be a matrix tuple in $\SCk^{\tilde{g}}$, and let $C_A$ be the nc smallest $A$-set. Suppose a Hermitian matrix-valued nc function $F$ in noncommuting tuples $a$ and $x$ is both matrix convex in $x$ and $x$-real entire at $(C_A,0)$. Then  
		$$F(\alpha,X) = F_0(\alpha,X) + F_1(\alpha, X) + F_2(\alpha, X)$$
	for each $\alpha \in C_A$ and all $X \in \mathbb H(\C^{\kappa m \times \kappa m})^g$, where $F_i$ are nc polynomials of degree $i$ in $X$.
\end{theorem}

\begin{proof}
See Section \ref{mainproof}.
\end{proof}

The special case which involves only $x$ variables is:

\begin{corollary}
Suppose that a Hermitian, matrix-valued nc function $F$ is both matrix convex and real entire near the origin. Then
	$$F(X) = F_0(X) + F_1(X) + F_2(X)$$
for all $X \in \SC$, where $F_i$ are matrix-valued nc polynomials of degree $i$ in $X$.
\end{corollary}

\begin{corollary}\label{maincor}
Let $\mathcal U \subset \mathbb H (\C)^{\tilde{g}}$ be an nc open set and let $F: \mathbb H(\C)^{\tilde{g}} \times \SC \to \mathbb H (\C)$ be a matrix-valued nc function in $a$ and $x$. Suppose that for each $A \in \U$, it is the case that $F$ is both matrix convex in $x$ and $x$-real entire at $(C_A,0)$. Then $F$ is a matrix-valued nc polynomial of degree $2$ or less in $X$.
\end{corollary}

\begin{proof}

By Theorem \ref{main}, for each $A \in \mathcal U$, we know that $F_k(A, X) = 0$ for all $(A, X)$ in an nc open set. Since $F_k$ is a polynomial, it is zero on all tuples $(A, X)$ of complex Hermitian matrix tuples.
Proposition A.7 in \cite[Appendix A]{helmcvin06} uses Rowen \cite{rowen80} to show that two rational functions $r_1$ and $r_2$ with real coefficients determine the same formal power series if and only if the functions agree for each $n$ on $\mathbb S(\R^{n\times n})^{\tilde{g} + g}$, the set of all symmetric $n \times n$ matrix tuples with real entries. Write $F_k= R_k  + i T_k$, where $R_k $ and $T_k$ are matrix polynomials with real coefficients. Since each $F_k$ is zero on $\mathbb H(\C^{n\times n})^{\tilde{g} + g}$ for each $n$ for $k>2$,
 and since $\mathbb S(\R^{n\times n})^{\tilde{g} + g} \subset \mathbb H(\C^{n \times n})^{\tilde{g} + g}$, consequently \cite[Proposition A.7]{helmcvin06}
 implies that $R_k =0$ and $T_k=0$ on $\mathbb S(\R^{n\times n})^{\tilde{g} + g}$, so $F_k=0$ for $k >2$.

\end{proof}

\section{Proofs}\label{proofs}

\subsection{The One-variable Proof}
The proof in the case of a function of one variable follows from the connection between matrix convexity and matrix monotonicity and from the integral representation of Pick functions. 

\begin{lemma}\label{onedimension}
	Let $f:\R \to \R$ be Hermitian matrix convex on an interval $(a,b)$. If $f$ has an analytic continuation to some open set $U \subset \C$ containing the real axis, i.e. $f$ is real entire, then $f$ is given by a polynomial of degree no more than 2.
\end{lemma}

\begin{proof}
	Let $f$ be matrix convex on $(a,b)$. Without loss of generality, this interval can be assumed to be $\mathcal I = (-1, 1)$. Now suppose that $f$ has an analytic extension to an open set $U$ containing the real line. This set can be assumed to contain $\cc{z}$ for any $z \in U$, as some open subset of $U$ will have this property. As $f$ is real-valued on $\mathcal I$ and holomorphic on $U$, we have $f(z) = \cc{f(\cc{z})}$. Note that the function $\hat{f}(t) = f(t) - f(0)$ is matrix convex on $\mathcal I$ if and only if $f(t)$ is matrix convex on $\mathcal I$. Define a function $g$ by
		\beq \label{gdef} g(t) = \frac{\hat{f}(t)}{t}. \eeq
	As $\hat{f}$ is holomorphic on $U$ and $\hat{f}(0) = 0$, $g$ is holomorphic on $U$. As $f$ is real valued on $\mathcal I$, so is $g$, and thus $\cc{g(\cc{z})} = g(z)$ for all $z \in U$.

	By \cite[Corollary V.3.11]{bha97}, the restriction $g_{\mathcal I}(t)$ of $g(z)$ to $\mathcal I$ is operator monotone. 
Then by L\"owner's Theorem (see e.g. \cite[IX]{don74} or \cite[Theorem V.4.7]{bha97}), $g_{\mathcal I}(t)$ can be analytically continued on all of $\C$ off of the set $(-\infty, -1]\cup[1,\infty)$ to a function $g_{\mathcal I}(z)$ which is a Pick function, i.e. that maps the upper halfplane and the lower halfplane into themselves. 

	As $g(z)$ and $g_{\mathcal I}(z)$ agree on $\mathcal I$, the identity theorem gives $g(z) = g_{\mathcal I}(z)$ on $U \backslash ((-\infty, -1]\cup[1,\infty))$ and thus that $g_{\mathcal I}(z)$ gives an analytic continuation of $g(z)$ off of $U$ (that is to the rest of the complex plane), and for all $z \in \C, g(z) = \cc{g(\cc{z})}$.

	As a Pick function, (see e.g. \cite[II.2 Theorem 1]{don74}), $g$ admits a unique canonical integral representation
		\beq \label{nevrep} g(z) = \alpha z + \beta + \int_\R \left(\frac{1}{\la - z} - \frac{\la}{\la^2 + 1}\right) \,\,\dd\mu(\la), \eeq
	where $\alpha \geq 0, \beta \in \R$, and $\mu$ is a positive Borel measure with 
		$$\int_\R \frac{1}{\la^2 + 1}\,\dd\mu(\la) < +\infty.$$
	Let $(a,b) \subset \R$ be any real interval. The function $g(z)$ is real-valued on $(a, b)$ and continues through $(a, b)$ to the lower half-plane by reflection, as $g(z) = \cc{g(\cc{z})}$ for all $z$. By \cite[II.2 Lemma 2]{don74}, the measure $\mu$ in \eqref{nevrep} has no mass on $(a, b)$. Thus for any choice of interval $(a, b) \subset \R$, for all measureable $E \subset (a,b)$, we have $\mu(E) = 0$, and thus $\mu$ is the $0$ measure. It follows that $g(z) = \alpha z + \beta$, and therefore that 
		$$f(z) = \alpha z^2 + \beta z + f(0),$$
	i.e. $f$ is a polynomial of degree $2$ or less (as $\alpha$, $\beta$ could have been $0$ in equation \ref{nevrep}).
\end{proof}

\subsection{Proof of Theorem \ref{main}}\label{mainproof}

This subsection is devoted to proving the main theorem of the paper.

\begin{proof}
Let $F$ be an $N\times N$ matrix-valued nc function given by
	$$F(a,x) = \bbm F_{11}(a, x) & \cdots & F_{1N}(a,x) \\ \vdots & \ddots & \vdots \\ F_{N1}(a,x) & \cdots & F_{NN}(a,x) \ebm$$ for nc functions $F_{ij}$. 
	Suppose that $F$ is $x$-real entire at $(C_A,0)$ and matrix convex in $x$ at $(C_A,0)$ for some fixed $A \in \SCk^{\tilde{g}}$. As $F$ is $x$-real entire at $(C_A,0)$, by Section \ref{analytic} property \eqref{Xanal}, there exists an $\varepsilon_{1}>0$ such that $F$ has an $x$-power series 
		\beq\label{powerseriesAXproof} F(a,x) = \sum_i F_i(a,x), \eeq
	where the $F_i$ are matrix-valued nc polynomials homogeneous of degree $i$ in $x$, that for each $\alpha \in C_A$ converges on the $x$-ball of radius $\varepsilon_1$ about $(\alpha,0)$. As $F$ is matrix convex in $x$ at $(C_A,0)$, there exists $\varepsilon_2>0$ so that $F$ is matrix convex in $x$ at $(C_A,0)$ in an $x$-ball of radius $\varepsilon_2$. Let $\varepsilon = \min \{\varepsilon_{1}, \varepsilon_{2}\}$. Fix $X \in \SCk^g$ such that $\norm{X} < \frac{1}{2}\varepsilon$. \black

	Define a function $\Phi$ of one nc variable by introducing a matrix parameter $\xi$ via 
		$$\Phi^{A,X}(\xi) = F(A \otimes I, X_1\otimes \xi, \ldots, X_g \otimes \xi),$$
	where the dimension of the identity $I$ matches the dimension of $\xi$. As $(A, X) \in \SCk^{\tilde{g}} \times \SCk^g$, given a vector $v\in \C^{N\kappa}$, we can define a function into $\C$ by
		\beq \label{conj} \ph^{A,X}_v(\xi) = (v\ad \otimes I) \Phi^{A,X}(\xi) (v \otimes I).\eeq

	We wish to apply Lemma \ref{onedimension} to the function $\phax$ in \eqref{conj}. To do so, we first show that on the level of scalars, when $\xi = z$, the function $\phax(\xi)$ is real entire in the classical sense of a function of one complex variable. In this case, $\phax:\C \to \C$ is given by
		\beq \phax(z) = v\ad F(A, zX) v. \eeq
	Let $s \in \R$. Then $sX \in \SCk^g$, and so 
since $F$ is $x$-real entire, by Section \ref{analytic} \eqref{Xent}, there exists $\varepsilon_{A, sX}$ so that the power series \eqref{powerseriesAXproof} for $F$ has an analytic continuation as a function of $g\kappa^2$ complex variables to the set $\{(A, Z) : \norm{Z - sX} < \varepsilon_{A, X}\}$. Let $\delta_s = \varepsilon_{A, sX}\norm{X}^{-1}$. Then $\phax(z)$ is analytic on the complex $\delta_s$-ball about $s$. As $s \in \R$ was arbitrary, $\phax(z)$ is a real entire function of one complex variable, i.e. analytic on a complex domain containing the real axis in $\C$. 

	Let $0 < \delta < 1$. We now show that $\phax(z)$ is matrix convex on the interval $(1 - \delta, 1 + \delta)$. For any $T \in \mathbb H(\C^{m\times m})$ with 
		\beq \label{tcond} (1 - \delta)I_m \prec T \prec  (1 + \delta)I_m,\eeq
	we have
		$$\norm{T - I_m} < \delta.$$
	Then
	\begin{align*}
		&\norm{(X_1 \otimes T, \ldots, X_g \otimes T)} \\
		&= \norm{(X_1 \otimes (T - I_m + I_m), \ldots, X_g \otimes (T - I_m + I_m))} \\
		&= \norm{(X_1 \otimes (T - I_m), \ldots, X_g \otimes (T - I_m)) + (X_1 \otimes I_m, \ldots, X_g \otimes I_m)} \\
		&< \norm{X} \norm{T - I_m} + \norm{X} \\
		&< (1 + \delta) \frac{\varepsilon}{2} \\
		&< \varepsilon.
	\end{align*}

	This estimate holds uniformly for $\delta$ for each level $m$. For each $m$, note that $A\otimes I_m \in C_A$ (see equation \eqref{zclosure}). 
	Thus for all $m$, for $T, \tilde{T}$ of size $m$ satisfying \eqref{tcond},
	\begin{align}
		\phax(t T &+ (1-t)\tilde{T}) = (v\ad \otimes I) \Phi^{A,X}(t T + (1-t)\tilde{T}) (v \otimes I) \notag \\
		&=  (v\ad \otimes I) F(A \otimes I, X \otimes (t T + (1-t)\tilde{T})) (v \otimes I) \notag \\
		&= (v\ad \otimes I)  F(A \otimes I, t(X \otimes T) + (1-t)(X \otimes \tilde{T}))  (v \otimes I) \notag \\
		&\preceq (v\ad \otimes I)\bigg[ t F(A \otimes I, (X \otimes T)) \phantom] \notag \\ & \hspace{.5in} \phantom[+(1-t)F(A \otimes I, (X \otimes \tilde{T}))\bigg] (v \otimes I) \label{Fconvex} \\
		&= t\phax(T) + (1-t)\phax(\tilde{T}), \notag
	\end{align}
	where relation \eqref{Fconvex} follows from the Hermitian matrix convexity of $F$ in $x$ at $(C_A, 0)$. 
	We have shown that the function $\phax(z)$ is both real entire and matrix convex on the interval $(1 - \delta, 1 + \delta)$. Therefore, by Lemma \ref{onedimension}, $\phax(z)$ is a polynomial in $z$ of degree $2$ or less. 

	Near $(A, 0)$, the function $F$ has the convergent power series \eqref{powerseriesAXproof}. Recall that each $F_i$ is a matrix-valued nc polynomial homogeneous of degree $i$ in $x$. Then for $\abs{z} < \varepsilon$, the homogeneity of the entries gives
	\begin{align*}
		\phax(z) &= v\ad F(A, zX) v \\
		&= v\ad\left[ \sum_i F_i(A, zX)\right] v \\
		&= \sum_i (v\ad F_i(A, X) v) z^i.
	\end{align*}
	As $\phax$ is a polynomial in $z$ of degree $2$ or less, for $i > 2$, we have $v\ad F_i(A, X) v = 0$. 
	Since the previous equation holds for all $v \in \C^{N\kappa}$, we get 
		$$ F_i(A, X) = 0 \text{ for } i > 2. $$	
As $X$ was arbitrary in the $x$-ball of radius $\varepsilon/2$ about $(A, 0)$, we get that each polynomial $F_i = 0$ on an open $x$-ball near $(A, 0)$ and thus for all $X$. Thus, for all $X \in \SCk^g$, we have
		\beq \label{penultimate} F(A, X) = F_0(A, X) + F_1(A, X) + F_2(A, X). \eeq

Now let $\alpha$ be an arbitrary element in $C_A$, so that for some unitary $U$, we can write $\alpha = U\ad (I_m \otimes A) U \in \mathbb H(\C^{\kappa m \times \kappa m})^{\tilde{g}}$. Since $C_\alpha \subseteq C_A$, it follows that $F$ is matrix convex in $x$ and $x$-real entire at $(C_\alpha, 0)$. Following the same argument as above with $\alpha$ replacing $A$ gives that for each $\alpha \in C_A$,
$$F(\alpha, X) = F_0(\alpha, X) + F_1(\alpha, X) + F_2(\alpha, X)$$
for all $X \in \mathbb H (\C^{\kappa m \times \kappa m})^g$.  
\end{proof}

\newpage

\bibliography{references}

\bibliographystyle{alpha}

\newpage

\centerline{NOT FOR PUBLICATION}

\tableofcontents

\printindex 

\end{document}